\documentclass[10pt,a4paper,reqno]{amsart}
\usepackage{enumerate}
\usepackage{amsmath,amsthm,amssymb,amsfonts,latexsym}
\usepackage[gen]{eurosym}
\usepackage[german,english]{babel}
\usepackage{mathtools}
\usepackage{hyperref}
\usepackage{esint}
\usepackage{latexsym}

\usepackage{tikz}

\newtheorem{theorem}{Theorem}

\newtheorem{remark}[theorem]{Remark}

\newtheorem{example}[theorem]{Example}
\newcommand{\be}{\begin{equation}}
\newcommand{\ee}{\end{equation}}
\newcommand{\bea}{\begin{eqnarray*}}
	\newcommand{\eea}{\end{eqnarray*}}
\newcommand{\beq}{\begin{eqnarray}}
\newcommand{\eeq}{\end{eqnarray}}

\newcommand{\mV}{\mathsf V}

\newcommand{\mv}{\mathsf v}
\newcommand{\mE}{\mathsf E}
\newcommand{\me}{\mathsf e}

\newtheorem{prop}[theorem]{Proposition}
\newtheorem{cor}[theorem]{Corollary}

\usepackage{todonotes}

\title[A note on Ambarzumian's theorem]{A note on Ambarzumian's theorem for quantum graphs} 

\subjclass[2010]{}


\author[P.~Bifulco]{Patrizio Bifulco}
\author[J.~Kerner]{Joachim Kerner}

\address{Lehrgebiet Analysis, Fakult\"at Mathematik und Informatik, Fern\-Universit\"at in Hagen, D-58084 Hagen, Germany}
\email{patrizio.bifulco@fernuni-hagen.de}

\email{joachim.kerner@fernuni-hagen.de}

\date{\today}

\thanks{
}

\begin{document}
	
	\begin{abstract} Based on the main result presented in \cite{BifulcoKerner}, we derive \break Ambarzumian--type theorems for Schrödinger operators defined on quantum graphs. We recover existing results such as the classical theorem by Ambarzumian and establish some seemingly new statements, too. 
	\end{abstract}
	
\maketitle

Let $\mathcal{G} = (\mV,\mE)$ be a finite, connected, compact metric graph with vertex set $\mV$ and edge set $\mE$, see, for instance, \cite{Mugnolo} for a more detailed introduction. Note that we also assume, without loss of generality, that the graph has no loops. The graph Hilbert space is
$$L^2(\mathcal{G})=\bigoplus_{\me \in \mE} L^2(0,\ell_\me)\ ,$$
where $\ell_\me \in (0,\infty)$ denotes the edge length of the edge $\me \in \mE$. On $L^2(\mathcal{G})$, assuming -- as in \cite{BifulcoKerner} -- that
\[
q =(q_\me)_{\me \in \mE} \in L^\infty(\mathcal G) \cap \bigoplus_{\me \in \mE} C^\infty(0,\ell_\me)\ ,
\]
and $\sigma=(\sigma_\mv)_{\mv \in \mV} \in \mathbb{R}^{|\mV|}$, we introduce the self-adjoint Schrödinger operator
\begin{equation*}\label{SchrödingerOperator}
H_\mathcal{G}^{q,\sigma}=-\frac{\mathrm{d}^2}{\mathrm{d}x^2}+q(x)
\end{equation*}
associated with the quadratic form, 
$$h_\mathcal{G}^{q,\sigma}(f) =  \int_\mathcal{G} \vert f^{\prime} \vert^2\ \mathrm{d}x + \int_\mathcal{G} q \vert f \vert^2\ \mathrm{d}x+\sum_{v \in \mV}\sigma_v|f(v)|^2$$ 
with form domain $H^1(\mathcal{G}) := \bigoplus_{\me \in \mE} H^1(0,\ell_\me) \cap C(\mathcal{G})$; here $C(\mathcal{G})$ denotes the space of continous functions on $\mathcal{G}$. Note that functions $f \in C(\mathcal{G})$ are, in particular, continuous across vertices. Although this is not important for us in this paper, we shall recall that each $\sigma_\mv$ determines the so-called \textit{matching conditions} in a given vertex $\mv \in \mV$ of the graph. More explicitly, the sum of all normal derivates of functions (pointing into the adjacent edges) equals $\sigma_\mv$ times the value of the functions in the vertex $\mv \in \mV$ which is well-defined due to continuity of the functions across vertices. Whenever $\sigma_v=0$ one speaks of \emph{standard} matching conditions whereas $\sigma_v \neq 0$ refers to so-called \emph{$\delta$-coupling conditions}.

Since we only look at finite metric graphs, the operator $H_\mathcal{G}^{q,\sigma}$ has compact resolvent and therefore purely discrete spectrum; we denote its eigenvalues by
\[
\lambda_0^{q,\sigma} < \lambda_1^{q,\sigma} \leq \lambda_2^{q,\sigma} \leq \dots \rightarrow +\infty\ .
\]
As a main result of~\cite{BifulcoKerner}, the authors established the relation
\begin{equation}\label{Result}
\lim_{N \rightarrow \infty}\frac{1}{N}\sum_{n=1}^{N}(\lambda_n^{q,\sigma}-\lambda_n^{0})=\frac{1}{\mathcal{L}}\int_{\mathcal{G}}q\ \mathrm{d}x+\frac{2}{\mathcal{L}}\sum_{v \in \mV}\frac{\sigma_v}{\deg(v)}\ ,
\end{equation}
 where $\mathcal{L}=\sum_{e \in \mE} \ell_e$ denotes the \emph{total length} of the graph $\mathcal{G}$ and $\deg(\mv)$ is the \emph{degree} of a vertex $\mv \in \mV$. Here, $(\lambda_n^0)_{n \in \mathbb{N}_0}$ shall denote the eigenvalues of the operator $H^{q=0,\sigma=0}_\mathcal{G}:=H^{0}_\mathcal{G}$.

Using~\eqref{Result} and based on the well-known papers~\cite{Ambarz,Borg,Levinson,HochSL,HochLieb}, the classical question we would like to address in this note is the following: Assuming we know that the eigenvalues of a Schrödinger operator $H^{q,\sigma}_\mathcal{G}$ are identical to those of $H^{0}_\mathcal{G}$, is it possible to conclude that $q=0$ as well as $\sigma=0$? Or more generally, is it possible to reconstruct the potential and the matching conditions in the vertices of the graph by knowing the spectrum of $H_\mathcal{G}^{q,\sigma}$? The second question was already answered in the negative by Borg in~\cite{Borg} where he looked at Laplacians on intervals subject to Robin boundary conditions, see also~\cite{Levinson}. However, as already noted by Ambarzumian in~\cite{Ambarz}, assuming the eigenvalues of $H^{q,\sigma=0}_\mathcal{G}$ on an interval are exactly those of the Neumann Laplacian on this interval, then one necessarily has that $q=0$. As we shall see in the sequel, a key role is played by the fact that the lowest Neumann eigenvalue on a graph (that is, for $\sigma=0$) is zero, which was already noted by Borg.

In a first step, we provide a generalization of~\cite[Theorem~3.1]{Davies:2013}.
\begin{prop}\label{GenDavies} Assume that the lowest eigenvalue of $H_\mathcal{G}^{q,\sigma}$ is non-negative and that
$$\int_{\mathcal{G}}q\ \mathrm{d}x + \sum_{\mv \in \mV}\sigma_\mv\leq 0\ .$$
Then $q=0$ and $\sigma=0$.
\end{prop}
 \begin{proof} The proof follows the one given by Davies. One first observes that, inserting the function which is constant equal to one on $\mathcal{G}$ into the quadratic form $h_\mathcal{G}^{q,\sigma}(\cdot)$, the min-max principle yields the inequality
 $$0 \leq \lambda_0^{q,\sigma}\leq \frac{1}{\mathcal{L}}\left(\int_{\mathcal{G}}q\ \mathrm{d}x+\sum_{\mv \in \mV}\sigma_\mv\right)\leq 0$$
and hence we conclude $\lambda_0^{q,\sigma}=0$. 

The next step then consists of showing that $q\geq 0$: To prove this one uses the analyticity as well as the concavity of the eigenvalue curve $\lambda_0^{\tau q,\tau\sigma}$, $\tau \in \mathbb{R}$ (note that the ground state eigenvalue is also known to be non-degenerate, see~\cite{KurasovUniqueness}). Namely, since $\lambda_0^{q,\sigma}=0$ and $\lambda_0^{q=0,\sigma=0}=0$, one concludes that actually $\lambda_0^{\tau q,\tau \sigma}=0$ for all $\tau \in \mathbb{R}$. Now, assuming $q$ is negative somewhere, one chooses a suitable test function $g \in H^1(\mathcal{G})$ such that $h_\mathcal{G}^{\tau q,\tau \sigma}(g) < 0$ for some $\tau$ large enough, leading to a contradiction. 

In the same way, whenever $\sigma_\mv < 0$ for some $\mv \in \mV$, then one again picks a suitable test function $\hat{g} \in H^1(\mathcal{G})$ supported in the vicinity of the vertex $\mv$ (in a way that in particular $\int_\mathcal{G} q \vert \hat{g} \vert^2 \ \mathrm{d}x$ is small enough) and for which one has
$$h_\mathcal{G}^{\tau q,\tau \sigma}(\hat{g}) =  \int_\mathcal{G} \vert \hat{g}^{\prime} \vert^2\ \mathrm{d}x + \tau \left( \int_\mathcal{G} q \vert \hat{g} \vert^2\ \mathrm{d}x+\sigma_\mv|\hat{g}(\mv)|^2 \right) < 0$$
for $\tau$ large enough, also contradicting $\lambda_0^{\tau q,\tau\sigma}=0$. 

Now, from $q\geq 0$ we conclude
\begin{equation*}\label{EqProofDavies}
    \sum_{\mv \in \mV}\sigma_\mv\leq 0 \ .
\end{equation*} 
Since also $\sigma \geq 0$, this immediately yields $\sigma=0$ and consequently
$$\int_{\mathcal{G}}q\ \mathrm{d}x = 0 \ ,$$
implying $q=0$.
\end{proof}


Using Proposition~\ref{GenDavies}, we can write down an abstract version of Ambarzumian's theorem for quantum graphs which then subsequently allows for a simple proof of the classical theorem by Ambarzumian from~\cite{Ambarz}, including its version on graphs, see \cite{BKS,PKUnderstanding}.
\begin{theorem}[Abstract inverse spectral problem]\label{MainResult}
    Assume that $H_\mathcal{G}^{q,\sigma}$ is a Schrödinger operator on a metric graph $\mathcal{G}$ such that 
    \begin{align}\label{eq:abstract-condition-ambartsumian}
    \int_\mathcal{G} q\ \mathrm{d}x + 2 \sum_{\mv \in \mV} \frac{\sigma_\mv}{\deg(\mv)} = 0 \ \Longrightarrow \ \int_\mathcal{G} q\ \mathrm{d}x + \sum_{\mv \in \mV} \sigma_\mv \leq 0 \ .
    \end{align}
    Then, if $\lambda_{n}^{q,\sigma} = \lambda_n^0$ for all $n \in \mathbb{N}_0 := \{0,1,2,\dots\}$, where $\lambda_n^0$ are the eigenvalues of $H_\mathcal{G}^0$, one has that $q = 0$ and $\sigma=0$.
\end{theorem}
\begin{proof}
    If $\lambda_n^{q,\sigma} = \lambda_n^0$ for all $n \in \mathbb{N}_0$, it follows by \eqref{Result} that $\int_\mathcal{G} q \ \mathrm{d}x + 2\sum_{\mv \in \mV} \frac{\sigma_\mv}{\deg(\mv)} = 0$. Thus, by assumption, it follows that $\int_\mathcal{G}q \ \mathrm{d}x + \sum_{\mv \in \mV} \sigma_v \leq 0$ and by Proposition \ref{GenDavies}, it follows that $\sigma = 0$ and $q = 0$.
\end{proof}
 Note that in Theorem \ref{MainResult} it is enough to assume that the spectra are \emph{asymptotically close} in the sense that $\lambda_n^{q,\sigma} - \lambda_n^0 \rightarrow 0$, $n \rightarrow \infty$, as this implies the convergence of the averages in \eqref{Result}; this could be compared to~\cite[Theorem~8]{KurasovSuhrAsymp}. 
\begin{example}   
   The condition formulated in Theorem~\ref{eq:abstract-condition-ambartsumian} gives a rather simple criterion for a Schrödinger operator $H_\mathcal{G}^{q,\sigma}$ to have a spectrum different from that of the operator $H^{0}_\mathcal{G}$. This is for instance the case for the Laplacian $H_\mathcal{G}^{0,\sigma} =:-\Delta_\mathcal{G}$ on a $3$-star with a negative $\delta$-type condition at the central vertex and natural conditions on the outer vertices (see Figure~1).
\end{example}
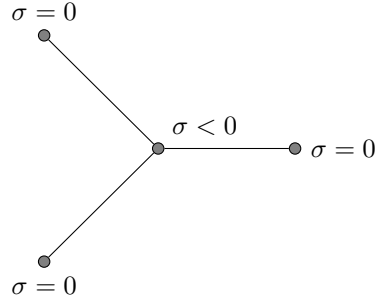
\begin{figure}[h]\label{fig:3-star}
		\begin{tikzpicture}[scale=0.60]
		\tikzset{enclosed/.style={draw, circle, inner sep=0pt, minimum size=.15cm, fill=gray}, every loop/.style={}}
		\node[enclosed, label={below: $\sigma=0$}] (V1) at (2.5,1.5) {};
		\node[enclosed, label={right: $\sigma=0$}] (V2) at (8,4) {};
  \node[enclosed, label={above: $\sigma=0$}] (V3) at (2.5,6.5) {};
		\node[enclosed, label={above right: $\sigma < 0$}] (C) at (5,4) {};
		
		\draw (V1) edge [] node[above] {} (C) node[midway, above] (edge1) {};
		\draw (C) edge [] node[below] {} (V2) node[midway, above] (edge2) {};
        \draw (C) edge [] node[below] {} (V3) node[midway, above] (edge3) {};
		\end{tikzpicture}
  \vspace{-0.5cm}
		\caption{$3$-star with central vertex $\mv_c$ and boundary vertices $\mv_1,\mv_2,\mv_3$}
	\end{figure}
As alluded to above, we list some direct consequences of  Theorem~\ref{MainResult}, leading to another proof of the classical Ambarzumian theorem, discussed in~\cite{Ambarz,Borg}, regarding Schrödinger operators on intervals and subject to Neumann boundary conditions, as well as another proof of Ambarzumian's theorem for Schrödinger operators on quantum graphs with standard matching conditions, meaning that $\sigma=0$, see~\cite{PivoAmba,BKS,YANG20161348,Pivovarchik2022,KissAmba}.
\begin{cor}[Ambarzumian theorem for $\delta$-coupling conditions on graphs]\label{cor:ambarzumian-delta-couplings}  Assume that $H_\mathcal{G}^{q,\sigma}$ is a Schrödinger operator defined over a metric graph $\mathcal{G}$ such that 
    $$\lambda_n^{q,\sigma} =\lambda_n^0\ , \quad n=0,1,\dots\ , $$
    where $\lambda_n^0$ are the eigenvalues of $H_\mathcal{G}^0$.  Then,
    whenever one of the following assertions is satisfied,
    \begin{itemize}
        \item[(i)] $\sum_{\mv \in \mV} \sigma_\mv \Big(1-\frac{2}{\deg(\mv)} \Big) \leq 0$,
        \item[(ii)] $\int_\mathcal{G} q \ \mathrm{d}x \geq 0$ and $\sigma \in [0,\infty)^{\vert \mV \vert}$,
        \item[(iii)] $\Big( \frac{\Delta_{\min}}{2}-1 \Big) \int_\mathcal{G} q \ \mathrm{d}x \geq 0$ and $\sigma \in (-\infty, 0]^{\vert \mV \vert}$, where $\Delta_{\min} := \min_{\mv \in \mV} \deg(\mv)$ denotes the minimal degree of the underlying graph $\mathcal{G}$,
    \end{itemize}
    one has that $q = 0$ and $\sigma=0$.
\end{cor}
\begin{proof} In order to employ Theorem~\ref{MainResult} it is sufficient to verify \eqref{eq:abstract-condition-ambartsumian}. 

If (i) holds, then this follows immediatly by calculation. Also, whenever (ii) holds, $\int_\mathcal{G} q \ \mathrm{d}x + 2\sum_{\mv} \frac{\sigma_\mv}{\deg(\mv)} = 0$ immediately gives $\int_\mathcal{G} q \ \mathrm{d}x = 0$ and $\sum_{\mv} \sigma_\mv=0$ and hence the desired inclusion. 

Finally, assume (iii) holds and that $\int_\mathcal{G} q \ \mathrm{d}x + 2\sum_{\mv} \frac{\sigma_\mv}{\deg(\mv)} = 0$. 
Then, multiplying both sides with the minimal degree $\Delta_{\min} := \min_{\mv \in \mV} \deg(\mv)$ it follows that
\[
\int_\mathcal{G} q \ \mathrm{d}x + \sum_{\mv \in \mV} \sigma_\mv \leq \frac{\Delta_{\min}}{2} \int_\mathcal{G} q \ \mathrm{d}x +  \sum_{\mv \in \mV} \frac{\Delta_{\min}}{\deg(\mv)} \sigma_\mv = 0
\]
and hence \eqref{eq:abstract-condition-ambartsumian} is indeed satisfied.
\end{proof}

To single out yet another case contained in Corollary~\ref{cor:ambarzumian-delta-couplings}, we call a graph a \textit{line graph} whenever $\deg(\mv) \in \{1,2\}$ for all $\mv \in \mV$. This is the case, for example, whenever the graph is actually an interval. Furthermore, we call a line graph a \textit{loop graph} whenever $\deg(\mv)=2$ for all $\mv \in \mV$.
\begin{cor}[Ambarzumian's theorem for loop and line graphs]  Assume that $H_\mathcal{G}^{q,\sigma}$ is a Schrödinger operator defined over a line graph such that
    $$\lambda_n^{q,\sigma} =\lambda_n^0\ , \quad n=0,1,\dots\ , $$
    where $\lambda_n^0$ are the eigenvalues of $H_\mathcal{G}^0$. Then, if either $\mathcal{G}$ is a loop graph or $\sigma \geq 0$ one has that $q = 0$ and $\sigma=0$.
\end{cor}
\begin{proof} This follows immediately from condition (i) in Corollary~\ref{cor:ambarzumian-delta-couplings}.
\end{proof}

\begin{remark} It is interesting to compare Theorem~\ref{MainResult} with some results of \cite{PKUnderstanding}. More explicitly, it can be shown that
the Schrödinger operator on the interval $I=[0,1/2]$ with potential $q(x)=\frac{2}{(1+x)^2}$ and $\sigma_0=-1$ and $\sigma_1=\frac{2}{3}$ has the same spectrum as the Laplacian subject to Neumann boundary conditions. So, in general and as elaborated on in \cite{PKUnderstanding}, Ambarzumian-type theorems cannot be expected to hold in utmost generality. However, note that this counterexample does not contradict Corollary~\ref{cor:ambarzumian-delta-couplings} or Theorem~\ref{MainResult}. Indeed, although one has
\begin{equation*}
    \int_\mathcal{G} q\ \mathrm{d}x + 2 \sum_{\mv \in \mV} \frac{\sigma_\mv}{\deg(\mv)}=  \int_0^{1/2} \frac{2}{(1+x)^2}\ \mathrm{d}x -\frac{2}{3} =0\ ,
\end{equation*}
one also has 
\begin{equation*}
    \int_\mathcal{G} q\ \mathrm{d}x +\sum_{\mv \in \mV} \sigma_\mv=  \int_0^{1/2} \frac{2}{(1+x)^2}\ \mathrm{d}x -\frac{1}{3}=\frac{1}{3} > 0\ .
\end{equation*}
\end{remark}

\subsection*{Acknowledgement}{PB was supported by the Deutsche Forschungsgemeinschaft DFG (Grant 397230547). We shall thank V.~Pivovarchik (Odesa) for providing useful references. We also enjoyed discussions with our colleagues as well as with our recent visitors to the FernUniversität in Hagen.
	
	\vspace*{0.5cm}
	
	{\small
		\bibliographystyle{amsalpha}
		\bibliography{Literature}}

\end{document}